\documentclass[12pt]{article}

\usepackage{verbatim}
\usepackage{amsthm}
\usepackage{qtree}
\usepackage{xytree}
\usepackage{lastpage}
\usepackage{fullpage}

\usepackage{amsfonts}
\usepackage{amsmath}
\usepackage{amsthm}
\usepackage{amssymb}
\usepackage{graphicx}
\usepackage{hyperref}
\usepackage{verbatim}
\usepackage{amsthm}
\usepackage{qtree}
\usepackage{xytree}

\usepackage[left=1in]{geometry}

\usepackage{fancyhdr}
\usepackage{lipsum}

\fancypagestyle{plain}
\newtheorem{mycor}{Corollary}
\newtheorem{mythm}{Theorem}
\newtheorem{mylem}{Lemma}

\newtheorem{myconj}{Conjecture}

\theoremstyle{definition}
\newtheorem*{mydef}{Definition}
\newtheorem*{exm}{Example}

\theoremstyle{remark}
\newtheorem*{myremark}{Remark}

\usepackage{times}

\title{
Lattice Variant of the Sensitivity Conjecture
}
\author{
Meena Boppana 
\footnote{This work was done under the direction of Dr. Scott Aaronson at MIT's Research Science Institute summer program.}\\
Hunter College High School\\
}

\date{
June 23, 2012
}

\begin{document}

\maketitle
\thispagestyle{empty}

\begin{abstract}
The Sensitivity Conjecture, posed in 1994, states that the fundamental measures known as the sensitivity and block sensitivity of a Boolean function $f$, $s(f)$ and $bs(f)$ respectively, are polynomially related.  It is known that $bs(f)$ is polynomially related to important measures in computer science including the decision-tree depth, polynomial degree, and parallel RAM computation time of $f$, but little is known how the sensitivity compares; the separation between $s(f)$ and $bs(f)$ is at least quadratic and at most exponential.  We analyze a promising variant by Aaronson that implies the Sensitivity Conjecture, stating that for all two-colorings of the $d$-dimensional lattice $\mathbb{Z}^d$, $d$ and the \emph{sensitivity} $s(C)$ are polynomially related, where $s(C)$ is the maximum number of differently-colored neighbors of a point.  We construct a coloring with the largest known separation between $d$ and $s(C)$, in which $d=O(s(C)^2)$, and demonstrate that it is optimal for a large class of colorings.  We also give a reverse reduction from the Lattice Variant to the Sensitivity Conjecture, and using this prove the first non-constant lower bound on $s(C)$.  These results indicate that the Lattice Variant can help further the limited progress on the Sensitivity Conjecture.
\end{abstract}

\section{Introduction}

We consider the following problem of two-coloring the $d$-dimensional lattice $\mathbb{Z}^d$.  The \emph{sensitivity} of a point in a two-coloring of $\mathbb{Z}^d$ is the number of differently-colored neighbors of that point, and the sensitivity of the coloring is the maximum sensitivity of all points on the lattice.  We say that a coloring is \emph{non-trivial} if the origin is colored red and there is at least one blue point on each axis.  The conjecture is that for any non-trivial coloring, the number of dimensions is always at most some polynomial of the sensitivity.  While this problem is of mathematical interest in itself, Aaronson showed that it implies the Sensitivity Conjecture.

The Sensitivity Conjecture was first posed in 1994 by Nisan and Szegedy.  The \emph{sensitivity} of a Boolean function is the number of bits in the input that, when flipped individually, change the output of the function.  The \emph{block sensitivity} is analogous, except that it is the largest number of disjoint subsets of bits such that flipping all the bits in a subset changes the output.  The Sensitivity Conjecture states that the sensitivity and block sensitivity of a Boolean function are \emph{polynomially related}--each is at most a polynomial function of the other.  In other words, the conjecture is that the separation between sensitivity and block sensitivity cannot be too large.  Nisan and Szegedy further conjectured that block sensitivity is at most a quadratic function of sensitivity, and the largest separation known to date is quadratic.  The best known upper bound on the separation, however, is exponential.

The block sensitivity of a function $f$ is known to be polynomially related to other important measures in computer science, such as the decision-tree complexity, certificate complexity, polynomial degree, and quantum oracle complexity of $f$.  Nisan \cite{Nisan CREW PRAM}  originally introduced block sensitivity to find the time needed to compute a boolean function on a parallel random access machine (PRAM), and used block sensitivity to show that the PRAM complexity is polynomially related to the decision tree complexity. The Sensitivity Conjecture, if true, would imply that the natural notion of sensitivity is related to this plethora of other measures, and would make it easier to show that new measures are polynomially related to block sensitivity.  

In Section 2, we set up Aaronson's Lattice Variant as well as the Sensitivity Conjecture, and summarize previous progress.  In Section 3, we construct a non-trivial coloring with a quadratic separation between the sensitivity and the number of dimensions, the largest known.  In Section 4, we extend Aaronson's reduction of the Sensitivity Conjecture to the Lattice Variant by providing a reduction in the other direction, mapping every coloring to a Boolean function.  In Section 5, we use the reverse reduction to establish the first non-constant lower bound on the sensitivity of non-trivial colorings in terms of the \emph{min-width} of the coloring.  Finally in Section 6, we show that our coloring from Section 3 achieves the optimal separation for the class of \emph{repeated} colorings, and describe a result by Palvolgyi that the coloring is optimal for all \emph{sliced} colorings.  

\section{Preliminaries and Previous Work}

\label{prelim}

\subsection{The Lattice Variant}

Consider a two-coloring of the $d$-dimensional lattice $\mathbb{Z}^d$, which is the set of all points $x$ in $d$ dimensions such that $x$ has all integer coordinates. Every point on the lattice is colored either red or blue. 

\begin{mydef}
A two-coloring $C$ of a $d$-dimensional lattice satisfies the \emph{non-triviality condition} if the origin is colored red and there exists at least one blue point on each of the $d$ axes.  We call such a coloring a \emph{non-trivial coloring}. 
\end{mydef}

For ease of notation, a non-trivial coloring is assumed to be in $d$ dimensions.

Let a \emph{neighbor} of a point $x$ in a coloring be a point $x'$ also on the lattice such that $|x'-x|=1$.  In other words, the neighbors of $x$ are the points which are 1 away from $x$ in a direction along an axis.

\begin{mydef}
The \emph{sensitivity of a point $x$} in a coloring $C$, denoted $s(C,x)$, is the number of neighbors of $x$ which are colored differently from $x$.  The \emph{sensitivity of a coloring $C$}, denoted $s(C)$, is the maximum of $s(C,x)$ over all $x$.
\end{mydef}

\begin{myremark}
The non-triviality condition guarantees that for colorings $C$ with $s(C) < d$, there are an infinite number of red and blue points.  This is because for any red point with sensitivity less than $d$, there must be another red point in one of the positive $x_1,\ldots,x_d$ directions, so one can move arbitrarily far away from the origin.  The same holds for blue points.   
\end{myremark}

\begin{myconj}  (Aaronson \cite{Aaronson: mathoverflow})
There exist constants $c$ and $k$ such that for all non-trivial colorings $C$, 
\[
d \leq c \cdot s(C)^{k}.
\]
\end{myconj}

We now present a slightly modified version of the sensitivity of a coloring, which we show is almost equivalent to the original definition.

\begin{mydef}
The \emph{axis-sensitivity} of a point $x$ in a coloring $C$, denoted $r(C,x)$, is the number of axes from $x$  along which there lies a differently-colored neighbor.  (Note that if there is a differently-colored neighbor both north and south of a point, it counts as one towards $r(C,x)$ and as two towards $s(C,x)$.)  We let $r(C)$ be the maximum over all points $x$ of $r(C,x)$.  
\end{mydef}

 It is clear that for all colorings $C$, $s(C) \leq r(C) \leq 2s(C)$, since there are between one and two differently-colored neighbors on each axis counting towards $r(C)$.

Furthermore, we show that any non-trivial coloring with $r(C)=n$ can be transformed into a non-trivial coloring $C'$ with $s(C')=n$.  Consider the coloring $C'$ constructed by replacing each point in the coloring $C$ with a hypercube of $2^d$ points all having the same color.  So every point $(x_1,\ldots,x_n)$ is replaced with the points of the form $(y_1,\ldots,y_n)$, where for all $i$, $2x_i-1 \leq y_i \leq 2x_i$.  In this case for all points $x$, $r(C',x)=s(C',x)$, since there is at most one neighbor along each axis from $x$.  Also $r(C',x)=r(C,x)$, so $s(C',x)=r(C,x)$ and $s(C')=r(C)$.  

This transformation conveniently allows us to use $r(C)$ instead of $s(C)$.

\vspace{.1 in}

Furthermore, the sensitivity of a coloring can be broken up into the red sensitivity and the blue sensitivity.
\begin{mydef}
Let the \emph{red sensitivity} of a coloring $C$, denoted $s^R(C)$ or $s^R$, be the maximum axis-sensitivity of all red points in $C$.  Let the \emph{blue sensitivity}, $s^B$, similarly be the maximum axis-sensitivity of all blue points.
\end{mydef}

\begin{figure}[ht]
\begin{center}
\includegraphics[scale=.4]{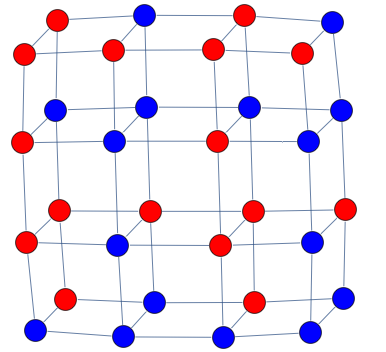}
\caption{A 3-dimensional coloring with $r(C)=2$.  Origin is in the lower-left-back corner, and figure is tessellated throughout space.}
\end{center}
\label{3d}
\end{figure}

It can be shown easily that 2 is the smallest value for the sensitivity of 3-dimensional colorings, and sensitivity 2 is achieved in Figure 1.

\subsection{The Sensitivity Conjecture}

An input $x$ is always assumed to be a string of length $n$ with bits $x_1,\ldots,x_n$.  
For an input $x$, let $x^i$ be the string with the $i$th bit flipped and all other bits intact.   Furthermore, if $B$ is a subset of  $\{1,2,\ldots,n\}$, we let $x^B$ denote the string $x$ with bits $i \in B$ flipped.   

We also refer to the all-zeroes input of a Boolean function as $\vec{0}$.  

\begin{mydef}
 Let $f: \{0,1\}^n \rightarrow \{0,1\}$ be a Boolean function.  The \emph{sensitivity of $f$ on $x$}, denoted $s(f,x)$, is the number of $i$, $1 \leq i \leq n$, such that $f(x) \neq f(x^i)$.  The \emph{sensitivity} of $f$, denoted $s(f)$, is the maximum over all inputs $x$ of $s(f,x)$. 
\end{mydef}  

For a subset $B$ of $\{1,2,\ldots,d\}$, $B$ is a \emph{sensitive block} if $f(x^B) \neq f(x)$.  

\begin{mydef}
The \emph{block sensitivity} of a Boolean function $f$ on $x$, denoted $bs(f,x)$, is the maximum number of disjoint sensitive blocks.  The \emph{block sensitivity} of $f$, denoted $bs(f)$, is the maximum over all inputs $x$ of $bs(f,x)$.   
\end{mydef}

\begin{exm}
Consider the following function that checks if the input bits are sorted, which has value 1 on the inputs 0000, 0001, 0011, 0111, 1000, 1100, 1110, or 1111 and 0 on the other inputs.  The sensitivity of this function is 2, which is achieved on the input 0000 among others.  The block sensitivity of this function is 3, however, which is achieved on the input (0)(1)(00) with sensitive blocks indicated in parentheses.
\end{exm}

The Sensitivity Conjecture is now the following: 
\begin{myconj} (Nisan and Szegedy \cite{Nisan and Szegedy})  There exist constants $c$ and $k$ such that for all Boolean functions $f$, $bs(f) \leq c \cdot s(f)^k$. 
\end{myconj}

Aaronson's reduction shows that for every Boolean function $f$, there exists a non-trivial coloring in $bs(f)$ dimensions with $s(C) \leq s(f)$, and so the Lattice Variant implies the Sensitivity Conjecture.   

\subsubsection{Kenyon and Kutin's Bound}

A \emph{minimal block} is a sensitive block $B$ such that no proper subset of $B$ is sensitive.  The sensitive blocks can always be chosen to be minimal without altering the block sensitivity.  We observe that the size of a minimal block $B$ of a function $f$ is at most $s(f)$.  If the block sensitivity is maximized on input $x$ and $B$ is a minimal block, then the input $x^B$ is sensitive on every bit in $B$.  

We now define $l$-block sensitivity, which provides an intermediate notion between sensitivity and block sensitivity.  

\begin{mydef}
The \emph{$l$-block sensitivity} of $f$ on input $x$, denoted $bs_l(f,x)$, is the maximum number of disjoint sensitive blocks of size at most $l$.  The \emph{$l$-block sensitivity of $f$}, $bs_l(f)$, is the maximum of $bs_l(f,x)$ over all $x$.
\end{mydef}

Note that for a function $f$ on $n$ input bits, $bs_1(f)=s(f)$ and, since all blocks can be made minimal, $bs_{s(f)}(f)=bs(f)$. Thus the $l$-block sensitivity lies between the two.  

Using $l$-block sensitivity, Kenyon and Kutin \cite{Kenyon and Kutin} provide an exponential upper bound on block sensitivity in terms of sensitivity.

\begin{mythm}  (Kenyon and Kutin \cite{Kenyon and Kutin})
\[
bs_l(f) \leq c_l s(f)^l ,
\]
where 
\[
c_l = \frac{(1+\frac{1}{l-1})^{l-1}}{(l-1)!} < \frac{e}{(l-1)!}.
\]
\end{mythm}

Since $bs(f)=bs_{s(f)} (f)$, it follows that $bs(f) < e^{s(f)+1} \sqrt{\frac{s(f)}{2\pi}}$ by Stirling's formula.

 \subsubsection{Optimal Function Constructions}
 
We also introduce the 0-sensitivity and 1-sensitivity of a function, which are analogous to the red and blue sensitivity of a coloring.
\begin{mydef}
The \emph{1-sensitivity}, $s^1(f)$, of a Boolean function $f$ is $\max_{x: f(x)=1} s(f,x)$.  Similarly,  the \emph{0-sensitivity}, $s^0(f)$, is $\max_{x: f(x)=0} s(f,x)$.
\end{mydef}

 Rubinstein \cite{Rubinstein} constructed a function $f$ on $n$ input bits such that $bs(f) = \frac{1}{2}s(f)^2$ in 1995.  Virza \cite{Virza} slightly improved on the separation by constructing a function $f$ with $bs(f)=\frac{1}{2}s(f)^2+s(f)$ in 2011.
 Ambainis and Sun \cite{Ambainis and Sun} recently showed that there is a function $f$ such that $bs(f)=\frac{2}{3} s(f)^2 - \frac{1}{2}s(f)$.  

Rubinstein's function is as follows.   For a fixed $n$, let $g(x)$ be a Boolean function on $n$ variables where $n$ is even.  We let $g(x)=1$ if there exists a $j$, $1 \leq j \leq \frac{n}{2}$, such that $x_{2j-1}=x_{2j}=1$ and $x_i=0$ for all $i \neq 2j-1,2j$.  
Define a function $f$ on $n^2$ variables to be the OR function of $n$ copies of $g$; in other words, $f(x)=1$ iff there is a string $y=x_{ik} x_{ik+1}... x_{i(k+1)-1}$ such that $g(y)=1$.  Then $bs^0(f) = \frac{1}{2}n^2$ and $s^0(f) = n$, so $bs(f)=\frac{1}{2}s(f)^2$.  
 
 Ambainis and Sun defined the following class of functions: functions $f$ which are the OR of some number of copies of a function $g$ with $s^0(g)=1$.  This class of functions includes Rubinstein's function, as well as their own function.  They showed that their function achieves the optimal separation of all functions up to a linear factor in this class.
 \begin{mythm} (Ambainis and Sun \cite{Ambainis and Sun}) 
The largest possible separation between $s(f)$ and $bs(f)$ in the class of functions defined above is $bs(f)=\frac{2}{3}s(f)^2+O(s(f))$.
\end{mythm}

\section{Coloring on the Lattice}
\label{constr}

We provide a coloring which achieves the largest known separation between the sensitivity and the number of dimensions, namely $d=2r(C)^2-r(C)$.  Our coloring has its blue points defined by \emph{slices}.
A slice is a hyperplane with one non-zero coordinate, coordinates which are fixed to be 0, and other free coordinates.

\begin{mydef}
A \emph{slice} is a set of points in $\mathbb{Z}^d$ satisfying the following property.  Given some subset $X$ of $\{1,2,\ldots,d\}$, a $y$ in $\{1,2,\ldots,d\}$ not in $X$, and a constant $c \neq 0$, a point $x$ is in the set iff $x_y=c$ and for all $i \in X$, $x_i=0$.
\end{mydef}

\begin{mythm}
There exists a non-trivial coloring $C$ with $d = 2r(C)^2-r(C)$.
\end{mythm}

\begin{table}
\begin{tabular}{ccccc|ccccc|ccccc}
\scriptsize\{1,1\}&\scriptsize\{1,2\}&\scriptsize\{1,3\}&\scriptsize\{1,4\}&\scriptsize\{1,5\}&\scriptsize\{2,1\}&\scriptsize\{2,2\}&\scriptsize\{2,3\}&\scriptsize\{2,4\}&\scriptsize\{2,5\}&\scriptsize\{3,1\}&\scriptsize\{3,2\}&\scriptsize\{3,3\}&\scriptsize\{3,4\}&\scriptsize\{3,5\} \\
\hline \
3&0&0&*&*&*&*&*&*&*&*&*&*&*&* \\
*&3&0&0&*&*&*&*&*&*&*&*&*&*&* \\
*&*&3&0&0&  *&*&*&*&*&  *&*&*&*&*\\ 
0&*&*&3&0&  *&*&*&*&*&  *&*&*&*&* \\
0&0&*&*&3&  *&*&*&*&*&  *&*&*&*&* \\
*&*&*&*&*&  3&0&0&*&*&  *&*&*&*&* \\
*&*&*&*&*&  *&3&0&0&*&  *&*&*&*&* \\
*&*&*&*&*&  *&*&3&0&0&  *&*&*&*&* \\
*&*&*&*&*&  0&*&*&3&0&  *&*&*&*&* \\
*&*&*&*&*&  0&0&*&*&3&  *&*&*&*&* \\
*&*&*&*&*&  *&*&*&*&*&  3&0&0&*&* \\
*&*&*&*&*&  *&*&*&*&*&  *&3&0&0&* \\
*&*&*&*&*&  *&*&*&*&*&  *&*&3&0&0\\
*&*&*&*&*&  *&*&*&*&*&  0&*&*&3&0\\ 
*&*&*&*&*&  *&*&*&*&*&  0&0&*&*&3\\ 
\end{tabular}
\caption{A table defining each of the 15 slices in the coloring for $d=15$ and $r(C)=3$, where each row is a slice.  Asterisks represent coordinates which can take on any value.}
\label{slices}
\end{table}

\begin{proof}

Extending a coloring by Palvolgyi \cite{Aaronson: mathoverflow}, we construct the following $(2n-1) n$-dimensional slices, each divided into $n$ groups of $2n-1$ coordinates each.  ($n$ is an integer which we will show is equal to $r(C)$.)  We define the $(2n-1) n$ coordinates ${\{i,j\}}$, where $1 \leq i \leq n$ and $1 \leq j \leq 2n-1$.   Let the slice $S_{\{a,b\}}$ consist of points $x$ such that $x_{\{a,b\}}=3$ and $x_{\{a,b+1\}},\ldots, x_{\{a,b+n-1\}}=0$, where subscripts $j$ of $x_{\{i,j\}}$ are evaluated modulo $2n-1$; the remaining coordinates can take on any value.  An illustration for the case $n=3$ is shown in Table \ref{slices}.

We now define the coloring where a point $x$ is colored blue if it is in a slice $S_{\{i,j\}}$ for some $i$ and $j$, and red otherwise.  This coloring satisfies the non-triviality condition.  The origin is colored red since it is not in any slice.  For any coordinate $\{a,b\}$, the slice $S_{\{a,b\}}$ contains the point  $y$ defined by $y_{\{a,b\}}=3$ and $y_{\{i,j\}}=0$ for all $(i,j)\neq (a,b)$.  This point $y$ is on the axis in the  $\{a,b\}$ direction and is blue.  

The blue sensitivity of this coloring is $n$ since, for any blue point $x$ in a slice $S_{\{i,j\}}$, only changing one of the $n$ coordinates $x_{\{i,j\}}, x_{\{i,j+1\}},\ldots, x_{\{i,j+n-1\}}$ yields an adjacent red point.  Furthermore, the red sensitivity is also $n$.  For a point $x$ and a slice $S_{\{i,j\}}$, if $x$ is adjacent to $S_{\{i,j\}}$ then $2 \leq x_{\{i,j\}} \leq 4$ and $-1 \leq x_{\{i,j+1\}},\ldots, x_{\{i,j+n-1\}} \leq 1$.  So for a point $p$ and an integer $a$ ($1 \leq a \leq n$), $x$ can be adjacent to at most one slice of the form $S_{\{a,j\}}$ because it cannot simultaneously satisfy two of these constraints, which are strings of $n$ coordinates, in its $2n-1$ coordinates.  Ranging over all $a$, it follows that a red point $y$ can be adjacent to at most $n$ slices.  Furthermore, a red point can be adjacent to at most one point per slice and the red sensitivity is at most $n$.  Therefore the coloring in fact has $r(C)=n$, and $d=2r(C)^2-r(C)$.
\end{proof}

\section{\Large Equivalence Between Boolean and Lattice Problems}
\label{equiv}

We describe Aaronson's reduction of the Sensitivity Conjecture to the Lattice Variant, which uses the fact that for every function $f$, there exists a coloring $C$ with $s(C) \leq s(f)$ and $d=bs(f)$.  We then show a reduction in the opposite direction, from a non-trivial coloring in $d$ dimensions to a Boolean function $f$, where $bs(f) \geq d$ and $s(f)$ is at most $s(C)$ times the \emph{min-width} of the lattice.

\begin{mydef}
The \emph{min-width} of a non-trivial coloring $C$ is the minimum integer $k$ such that there is a blue point within $k$ units of the origin for all of the $d$ axes.  
\end{mydef}

\begin{mythm} (Aaronson \cite{Hatami: variations})
Given a Boolean function $f$ and blocks $B_1,B_2,\cdots,B_d$, there exists a non-trivial coloring $C$ in $d=bs(f)$ dimensions with $s(C) \leq s(f)$.
\end{mythm}

\begin{proof}
Let the sizes of $B_1,B_2,\cdots,B_d$ be $b_1,b_2,\ldots,b_d$ respectively, i.e. $|B_i|=b_i$ for all $i$.
We will demonstrate a coloring of the infinite lattice which is periodic and repeats with every $2(b_1+1) \times 2(b_2+1) \times \ldots \times 2(b_d+1)$ hypercube.  

Without loss of generality, let $\vec{0}$ be an input on which the block sensitivity is maximized, and let $f(0)=0$.  (If not one can replace variables $a_i$ with $\bar{a_i}$.)  Now for all $i$, assign an order to the input bits in $B_i$.  

First we define the coloring for all points $m=(m_1,\ldots,m_d)$ where for all $i$, $0 \leq m_i \leq b_i$.  Construct the function input $x$ where for all $i$, the first $m_i$ bits in block $B_i$ are 1, and all other bits in $B_i$ are 0.  If $f(x)=0$, then we color the point $m$ red, and otherwise we color it blue.
Since the origin corresponds to the input $\vec{0}$ and $f(\vec{0})=0$, the origin is colored red.  Furthermore for all $i$ the point $(0,\ldots,0,b_i,0,\ldots,0)$, where $b_i$ is in the $i$th coordinate, corresponds to $\vec{0}^{B_i}$ (the input where bits in $B_i$ equal 1 and all other bits are 0) and $f(\vec{0}^{B_i})=1$, so the point is colored blue.  Thus the coloring satisfies the non-triviality condition. 

We now define a general lattice point $x=(x_1,x_2,\ldots,x_d)$ in the following way.  For all $i$, we compute the following quantities $z_i$ and $y_i$ from $x_i$.  Let $z_i = x_i \pmod{2(b_i+1)}$.  If $z_i > b_i$, let $y_i = 2b_i+1 - z_i$.  If $z_i \leq b_i$, let $y_i=z_i$.  Color $x$ the same color as $y=(y_1,y_2,\ldots,y_n)$ ($y$ is in the range of points whose color was previously defined).  This alternately tiles the original $(b_1+1) \times (b_2+1) \cdots \times (b_n+1)$ hypercube and its mirror image throughout the plane.   The sensitivity of a point $x$, where the corresponding $y_i$ satisfy $0 < y_i < b_i$, is at most $s(f)$.  This is because moving from $x$ to one of its neighbors corresponds to flipping a unique bit of the function $f$, and the color of the neighboring point is different iff the output of the adjacent function input flips.  The sensitivity of a point $x$ on the border of a hypercube ($y_i=0$ or $y_i=b_i$ for some $i$) is also at most $s(f)$.  The neighbors of $x$ which are in a different hypercube are all the same color as $x$ by the fact that the neighboring hypercubes are mirror images. Neighbors of $x$ within the hypercube also correspond to flipping a bit of $f$ and so the sensitivity of $x$ is at most $s(f)$ as well.   Therefore, $s(C)$ is at most $s(f)$.

\end{proof}

\begin{mythm}
\label{converse}
Given a non-trivial coloring $C$ with min-width $k$, there exists a Boolean function $f$ such that $bs(f) \geq d$ and $s(f) \leq k \cdot s(C)$.  
\end{mythm}

\begin{proof}
There exists a blue point on each axis by the non-triviality property.  Without loss of generality, let there be at least one blue point on the \emph{positive} $i$ axis for all $i$, or else the coloring can be reflected about the $i$th axis.  Let $b_i$ be the smallest positive integer such that the point $(0,\ldots,0,b_i,0,\ldots,0)$ is blue, where $b_i$ is in the $i$th coordinate.  We define a function $f$ on $n=\Sigma_{i=1}^d b_i$ bits as follows.  Divide the bits into blocks $B_1, B_2, \ldots, B_d$, where $|B_i| = b_i$ for all $i$.  For a function input $y=(y_1,y_2,\ldots,y_n)$, let $z_i$ be the number of 1's in block $B_i$ for all $i$.  Letting $y$ correspond to the lattice point $(z_1,z_2,\ldots,z_d)$, we define $f(y)=0$ if $(z_1,z_2,\ldots,z_d)$ is red and $f(y)=1$ if it is blue.  Note that $f$ is somewhat symmetric in that the bits in any one block can be permuted without changing the output.

Each of the blocks is sensitive on the function input $\vec{0}$.  This is because flipping the bits in block $B_i$ corresponds to the blue point $(0,\ldots,0,b_i,0\ldots,0)$ where $b_i$ is the $i$th coordinate,  and so $f(y^{B_i})=1$.  Consequently, $bs(f) \geq d$.

For any point $x=(x_1,x_2,\ldots,x_n)$ let $X \subseteq \{1,2,\ldots,d\}$ be the set of axes from $x$ along which there lies a differently-colored neighbor of $x$.  Moving one unit along the $i$th axis from $x$, either in the positive or negative direction, corresponds to changing one of the bits in $B_i$ of the corresponding function input. 
 Since $|X| $ is at most $s(C)$ and $|b_i| $ is at most the min-width $k$, the sensitivity of $f$ on the corresponding input is at most
\[
\sum_{i \in X} b_i \leq s(C) \cdot k.
\] 
\end{proof}

\section{Lower Bound on the Lattice}
\label{lower}

We prove a lower bound on the sensitivity of all non-trivial colorings in terms of the min-width of the coloring and the number of dimensions. 

\begin{mythm}
For a non-trivial coloring of min-width $k$ and $s(C)=s$, 
\[
s \geq \alpha \cdot d^\frac{1}{k},
\]
where $\alpha = \frac{1}{e^2}$.
\end{mythm}

\begin{proof}

The min-width $k$ is a positive integer.  If $k=1$, then there is a blue point 1 unit away from the origin on every axis.  So the sensitivity of the origin, which is red, is at least $d$ and $s \geq \frac{1}{e^2} \cdot d$ as desired.  

Assume that $k > 1$.   A coloring with min-width $k$ and sensitivity $s$ can be reduced to a function $f$ with $bs(f)\geq d$ and $s(f)\leq ks$ by Theorem \ref{converse}. 
Because the sensitive blocks in $f$ have size at most $k$, we see that $bs_k(f) \geq d$ as well ($bs_k(f)$ is the $k$-block sensitivity defined in Section 2.2.1).  Thus applying Kenyon and Kutin's \cite{Kenyon and Kutin} result in Theorem 1 that $bs_k(f) \leq c_k \cdot s(f)^k$ yields
$d \leq bs_k(f) \leq c_k (ks)^k$.
Solving for $s$ shows that
\[
s \geq \frac{1}{k} \left( \frac{d}{c_k}\right)^{\frac{1}{k}}.
\]
Furthermore, since $c_k < \frac{e}{(k-1)!}$, we get that
\[
s \geq \frac{1}{k} \left( \frac{d (k-1)!}{e}\right)^{\frac{1}{k}}.
\]

We now show that $\frac{1}{k} \left( \frac{d (k-1)!}{e}\right)^{\frac{1}{k}} \geq \alpha \cdot d^\frac{1}{k}$.  Since $e \geq (1+\frac{1}{x})^x$ for all positive $x$ and $e^k \geq ke$, we get that
\[
\left( \frac{1}{e\alpha} \right)^k = e^k \geq ke \geq k(1+\frac{1}{k-1})^{k-1}.
\]
Manipulation shows that $(k-1)^{k-1} \geq k^k e^k \alpha^k$ and $\left( \frac{k-1}{e} \right) ^{k-1} \geq k^k e \alpha^k$.  Applying Stirling's formula, which says that $k! \geq \left( \frac{k}{e} \right)^k$ for all $k$, shows that $(k-1)! \geq k^k e \alpha^k$.  This yields that $\frac{1}{k} (\frac{(k-1)!}{e}) \geq \alpha^k$, and finally $\frac{1}{k} \left( \frac{(k-1)!}{e} \right)^{\frac{1}{k}} \geq \alpha$.  So
\[
s \geq \frac{1}{k} \left( \frac{d(k-1)!}{e} \right)^{\frac{1}{k}} \geq \alpha d^\frac{1}{k}
\] 
as desired.

\end{proof}

\section{Optimality of Our Coloring}
\label{optimal}

We show the coloring described in Section \ref{constr} achieves the greatest possible separation for the class of \emph{repeated} colorings.  We also describe a result by Palvolgyi which shows that $d \leq 2r(C)^2-r(C)$ for all \emph{sliced} colorings.  Since our coloring from Section 3 is both a repeated and sliced coloring, this coloring obtains the optimal separation between $d$ and $r(C)$ in both classes.

\subsection{Repeated colorings}

A repeated coloring is a coloring which is the $n$-fold Cartesian product of a coloring with $s^R=1$ (the red sensitivity).  The class of repeated colorings is analogous to the class of Boolean functions analyzed by Ambainis and Sun \cite{Ambainis and Sun}. Using a lemma by Palvolgyi et al. that  $d \leq 2s^B-1$ for colorings with $s^R=1$, we show that $d \leq s^R(2s^B-1) \leq 2r(C)^2-r(C)$ for all repeated colorings.  

Let $C$ be a coloring in $nk$ dimensions.  We define the $nk$ coordinates ${\{i,j\}}$, where $1 \leq i \leq n$ and $1 \leq j \leq k$. For a point $x$ with coordinates $x_{\{i,j\}}$, let $x_i$ be the $k$-tuple consisting of coordinates $x_{\{i,j\}}$ for all $1 \leq j \leq k$.

\begin{mydef}
$C$ is a \emph{repeated} coloring in $nk$ dimensions if there exists a non-trivial coloring $C'$ in $k$ dimensions such that:
\begin{enumerate}
\item $s^R(C')=1$
\item A point $x$ is colored blue in $C$ iff there exists an $i$, $1 \leq i \leq n$, such that $x_i$ is blue in $C'$.
\end{enumerate}
\end{mydef}
We assume that all repeated colorings $C$ are in $d=nk$ dimensions, with a corresponding $C'$ in $k$ dimensions.

\begin{myremark}
Our coloring from Section \ref{constr} is a repeated coloring.  The coloring is a Cartesian product of $n-1$ copies of a coloring $C'$ in $n$ dimensions.  Furthermore $s^R(C')=1$.  The $n$-dimensional coloring $C'$ consists of all slices of the form $S_{\{1,j\}}$ (with all but the first $n$ coordinates deleted) and as shown in Theorem 3, a red point can only be adjacent to at most one slice of the form $S_{\{a,j\}}$ for an integer $a$.  
\end{myremark}

\begin{mylem}
\label{concat}
For all repeated colorings $C$ in $nk$ dimensions, we have $s^R(C) = n$.  
\end{mylem}

\begin{proof}
Consider a red point $x$ in $C$.  For all $i$, $r(C',x_i) \leq 1$.  Since $x$ is blue in $C$ if $x_i$ is blue in $C'$ for some $i$, there is at most one coordinate in each $x_i$ which contributes to the sensitivity of $x$.  Therefore $r(C,x) \leq n$, and $s^R(C) \leq n$.  

Furthermore, $s^R(C) \geq n$.  Let $y$ be a red point in $C'$ with $r(C',y)=1$, and let $z$ be the concatenation of $n$ copies of $y$.  The sensitivity of $z$ is $n$, so $s^R(C)=n$.
\end{proof}

\begin{mylem}
\label{same}
For all repeated colorings $C$, we have $s^B(C) = s^B(C')$.  
\end{mylem}

\begin{proof}
Consider a blue point $x$ in $C$.  If more than one $x_i$ is blue in $C'$ then $s^B(C)=0$, since changing one coordinate of $x$ will not change all the blue $x_i$ to red.  So, the sensitivity of $x$ is maximized when exactly one $x_i$ is blue in $C'$, and in this case $r(C,x) = r(C',x_i)$.  Therefore $s^B(C)=s^B(C')$.
\end{proof}

\begin{mythm}
For all repeated colorings $C$, we have $d \leq s^R(C) \cdot (2s^B(C)-1)$.  
\end{mythm}

\begin{proof}
We use a lemma
by Palvolgyi et al. \cite{Aaronson: mathoverflow} that for all non-trivial colorings with $s^R=1$, $d \leq 2s^B-1$.
By this lemma, $k \leq 2s^B(C')-1$.  Applying Lemma \ref{same}, we get that $k \leq 2s^B(C)-1$.
Furthermore, $kn \leq n \cdot (2s^B(C)-1)$ and by Lemma \ref{concat}, we get that
\[
d \leq s^R(C) \cdot (2s^B(C)-1).
\]
Since $s^R$ and $s^B$ are at most $r(C)$, we obtain that $d \leq 2r(C)^2 - r(C)$ as desired.
\end{proof}

Therefore, our coloring achieves the largest possible separation of all repeated colorings, since it satisfies $d=2r(C)^2-r(C)$.

\subsection{Sliced colorings}

\begin{mydef}
A non-trivial coloring $C$ in $d$ dimensions is a \emph{sliced coloring} if the set of blue points is the union of exactly $d$ slices with non-zero coordinate (the coordinate in a slice which is fixed to be a non-zero constant) at least 3.
\end{mydef}

$C$ is made up of $d$ slices, and by the non-triviality condition, one slice must intersect each axis.  So the $d$ slices each have their non-zero coordinates in a different dimension, since a slice intersects the $i$th axis iff its non-zero coordinate is in the $i$th dimension.  Let the slice with its non-zero coordinate in the $i$th dimension be $S_i$.

\begin{myremark}
The coloring described in Section \ref{constr} is a sliced coloring.
\end{myremark}

\begin{mylem}  (Palvolgyi \cite{Aaronson: mathoverflow})
\label{intersect}
The red sensitivity of a sliced coloring $C$ is at least the maximum number of slices which intersect at one point.
\end{mylem}

\begin{proof}

Let $n$ be the maximum number of slices which intersect.  Suppose that $n$ slices labelled $S_{k_1}, \ldots, S_{k_n}$ intersect at some point $x$ (so $S_{k_i}$ has its non-zero coordinate in the $k_i$th coordinate).  Define the point $x'$, where $x'_{k_i} = x_{k_i} +1$ for all $i$ from 1 to $n$, and $x'_j=x_j$ for all $j$ not equal to $k_i$ for some $i$.  Then $x'$ is red and has sensitivity $n$, since decreasing any coordinate $x'_{k_i}$ yields a blue point, so $s^R \geq n$.   

\end{proof}

\begin{mythm} (Turan) 
A graph induced on $n$ vertices with average degree at most $k$ has an independent set of size at least $\frac{n}{k+1}$.
\end{mythm}

\begin{mythm}
\label{slicedthm}
(Palvolgyi \cite{Aaronson: mathoverflow})   For all non-trivial sliced colorings $C$, $d \leq s^R(2s^B-1)$.
\end{mythm}

\begin{proof}
For all $S_i$, $1 \leq i \leq d$, let $B_i$ be the set of coordinates which are fixed to be 0 in the slice.  

We note that there can be at most $s^B$ fixed coordinates of a slice including the non-zero coordinate, since changing a fixed coordinate contributes to the sensitivity of a blue point in the slice.  So $|B_i| \leq s^B-1$. 

Consider the following directed graph $G$ on $d$ vertices, labelled 1 through $d$.  Draw a directed edge from $i$ to $j$ iff $i \in B_j$.  Since $i \not\in B_i$, there are no loops in the graph.  There is a directed edge from $i$ to $j$ iff $S_i$ and $S_j$ \emph{don't} intersect.  This is because if there is an edge from $i$ to $j$ then the $i$th coordinate is non-zero in $S_i$ and 0 in $S_j$, and conversely if there is no edge then the slices intersect at the point $(0,\ldots,0,c,0,\ldots,0)$, where $c$ is the non-zero coordinate of $S_i$ in the $i$th coordinate.

We claim that the maximum size of an independent set in this graph is $s^R$.    Assume to the contrary that there is an independent set of size $s^R+1$.  For any $i$ and $j$ in the independent set, $B_i$ and $B_j$ intersect.  If two slices intersect and a third slice intersects both of them, then all three slices mutually intersect.
The $s^R+1$ slices corresponding to vertices in the independent set all pairwise intersect, and so they all mutually intersect.  This is a contradiction of Lemma \ref{intersect} that at most $s^R$ slices can mutually intersect at a point.

We now show that $d \leq s^R(2s^B-1)$ by applying Turan's theorem.  Consider the corresponding undirected graph, $G'$, of $G$.  
Note that outdegree of a vertex in $G$ is at most $s^B-1$ since $|B_i| \leq s^B-1$, and thus the degree of a vertex in $G'$ is at most $2s^B-2$.  Clearly the average degree is at most $2s^B-2$, and so by Turan's theorem there is an independent set of size at least $\frac{d}{2s^B-1}$.  However, the size of an independent set can be at most $s^R$, so $s^R \geq \frac{d}{2s^B-1}$ and $d \leq s^R(2s^B-1)$.
\end{proof}

\begin{mycor}
For all non-trivial sliced colorings $C$, it holds that $d \leq 2r(C)^2-r(C)$.
\end{mycor}

Therefore, our coloring achieves the maximum separation in the class of sliced colorings, as well as in the class of repeated colorings.  In order to construct a coloring with a larger separation, one would need to resort to considerably different methods and obtain a function outside of both of these classes.

\section{Acknowledgments} 
The author thanks Dr. Scott Aaronson at MIT for his guidance and for introducing her to this research problem.  She thanks the Center for Excellence in Education (CEE) for sponsoring the Research Science Institute (RSI) program at MIT, as well as the staff at RSI.

\newpage

\end{document}